\documentclass[times]{biomat_en}
\usepackage{natbib,epsfig,graphicx,graphics,amssymb,amsfonts,newlfont,float,indentfirst,fancyvrb,url,textcomp,subfigure}
\usepackage[centertags]{amsmath}
\usepackage{amsthm}
\usepackage[english]{babel}
\newcommand{\bb}{\begin{equation}}
\newcommand{\ee}{\end{equation}}

\newtheorem{theorem}{Theorem}[section]
\newtheorem{lemma}[theorem]{Lemma}

\theoremstyle{definition}
\newtheorem{definition}[theorem]{Definition}
\theoremstyle{definition}
\begin{document}

\title{State Feedback as a Strategy for Control and Analysis of COVID-19}
\author
 {Leonardo R. S. Rodrigues\thanks{leonardo.rodrigues@ufma.br}\,,\\
    Natural Sciences, CCCO -- UFMA, 65.400-00, Cod\'o/MA.\\[3mm]
  Felipe Gabrielli\thanks{gabrielli14@gmail.com},\,\,
  \\
    Natural Sciences, CCCO -- UFMA, 65.400-00, Cod\'o/MA.}

\criartitulo

\runningheads{Rodrigues, L. \& Gabrielli, F.}{State Feedback as a Strategy for Control and Analysis of COVID-19}

\begin{abstract}

{\bf Abstract}. This paper presents a study on a compartmental epidemic model for COVID-19, examining the stability of its equilibrium points upon the introduction of vaccination as a strategy to mitigate the spread of the disease. Initially, the SIQR (Susceptible-Infectious-Quarantine-Recovered) mathematical model and its technical aspects are introduced. Subsequently, vaccination is incorporated as a control measure within the model scope. Equilibrium points and the basic reproductive number are determined, followed by an analysis of their stability. Furthermore, controllability characteristics and Optimal Control strategies for the system are investigated, supplemented by numerical simulations.
\end{abstract}


\Section{Introdution}
Mathematical modeling has long been used as a tool in several areas of public health, including epidemiology, an area that developed significantly during the 20th century. Models in mathematical epidemiology, in particular, have been studied since the 18th century but had a development leap arguable from the work of Kermark and Mcckendrick \cite{kermark}, from 1927. Since Then, many other advances were made and many types of models were created and studied. Most of these models are compartmental models, that divide the population into categories with a particular behavior, some examples are SIS, SIR, SIRS, SEIR, and SIQR, among others. More details about some of these models can be seen in \cite{Brauer}. One important point to observe is that a mathematical model is always a simplification of reality, some aspects are disregarded so we can focus on the variables that really matter to the problem, no model can consider all aspects of a complex real problem, like the spread of infectious disease, hence the importance of each model type. Just to exemplify this reality simplification, the model studied in this paper does not consider population heterogeneity,  that is, all individuals are equally susceptible to the disease, there are models that take these differences into account as can be seen in \cite{science}.

We are currently living through the COVID-19 pandemic, a disease caused by the sars-cov-2 virus, which has already caused thousands of deaths around the world and continues to plague the population. Many researchers believe that COVID-19 will become an endemic disease in the future, but until the present date World Health Organization keeps classifying the threat level of COVID-19 as a pandemic, an interesting discussion about predicting the contention of the pandemic is done in \cite{rep_numb}. Several strategies have been adopted by governments to combat the spread of the disease, such as quarantine, lockdown, closing borders, use of masks, hand hygiene with alcohol gel, etc. But no measure is as effective as vaccination and since its development, in 2021, many countries have implemented a vaccination calendar as part of the disease combat strategies. Vaccination is the most effective and safe way we know to combat infectious diseases and it was responsible, for example, for eradicating smallpox.

Our objective in this work is to consider and analyze the properties of the SIQR model by adding vaccination as a strategy to control the growth of the disease, study the stability of constant solutions, calculate the basic reproductive number of disease propagation, study controllability of the system and the conditions for we obtain the Optimal Control and apply the model in some numerical simulations (using MATLAB\textsuperscript{\textregistered} software) to reach some conclusions about the control method (vaccination).

This analysis, via theoretical modeling,  is very important to complement the models that work with empirical data, to compare, complement and adjust possible strategies to face the disease, given that empirical data is not entirely trustful as can be seen in \cite{unreported} due to the number of unreported cases and the low number of tests in many countries.

\section{SIQR Mathematical Model}

In this section we will describe the mathematical model used to study the spreading of COVID-19, its elements, and technical features.

\bigskip

In order to study the spreading of infectious disease we have to consider a population whose size varies with time, we represent this population as $N(t)$. The mathematical model we will use is a SIQR compartmental model that divides the total population into four groups: susceptible ($S(t)$), infected ($I(t)$), quarantined ($Q(t)$) and recovered (sometimes also called removed) ($R(t)$), thus:

\begin{equation}\label{pop}
	N(t) = S(t) + I(t) + Q(t) + R(t)
\end{equation}

\bigskip

The model is the following system of ordinary differential equations:

\begin{equation}\label{sist_bas}
	\begin{cases}
		\frac{dS}{dt} = \Delta - \alpha SI - \mu S\\
		\frac{dI}{dt} = \alpha SI - (\gamma + \mu + \eta)I\\
		\frac{dQ}{dt} = (\eta - \epsilon)I - (\rho + \mu)Q\\
		\frac{dR}{dt} = \gamma I + \rho Q - \mu R
	\end{cases}\,.
\end{equation}

Such that:

\begin{table}[ht]
	\centering
	\begin{tabular}{cl}
		\hline
		\textbf{Variable} & \textbf{Definition} \\
		\hline
		$T$ & Time  \\
		$\alpha$ & Effective contact rate between susceptible and infectious class \\
		$\gamma$ & Natural recovery rate \\
		$\mu$ & Natural death rate \\
		$\rho$ & Removed rate from quarantine to recovered \\
		$\epsilon$ & Disease-related death rate \\
		$\eta$ & Quarantine rate of the infectious class\\
		$\Delta$ & Recruitment rate of the population\\
		\hline
	\end{tabular}
	\caption{\label{tab_param}Model Parameters.}
\end{table}

The model works with the following dynamics: Each compartment has an initial portion of the population $N$, if we want to simulate the beginning  of the pandemic, for example, we can put $Q=0$, $R=0$ and even $I=1$ (representing patient zero). After that, the parameters will change the quantities of the population in each compartment at each time interval, adding or subtracting some portions. In the first equation, the susceptible population is increased by $\Delta$, then some portion (determined by $\alpha$) is subtracted from the susceptible and added to the Infectious group, still in the first equation another portion is subtracted due to $\mu$, the natural rate of death. The population of the second group is diminished as well by $\mu$ and by $\gamma$ and $\eta$, natural recovery rate and quarantine rate, respectively. The third and fourth equations follow the same dynamics,  in the third a portion (due to $\epsilon$) is subtracted representing the disease death rate and another portion is subtracted from there and added to the recovered group ($\rho$) representing the quarantine recovery rate.

\bigskip

In figure \ref{fig:fluxo} we have a schematic diagram showing the dynamics of the model:

\begin{figure}[H]
	\centering
	\includegraphics[width=0.8\textwidth]{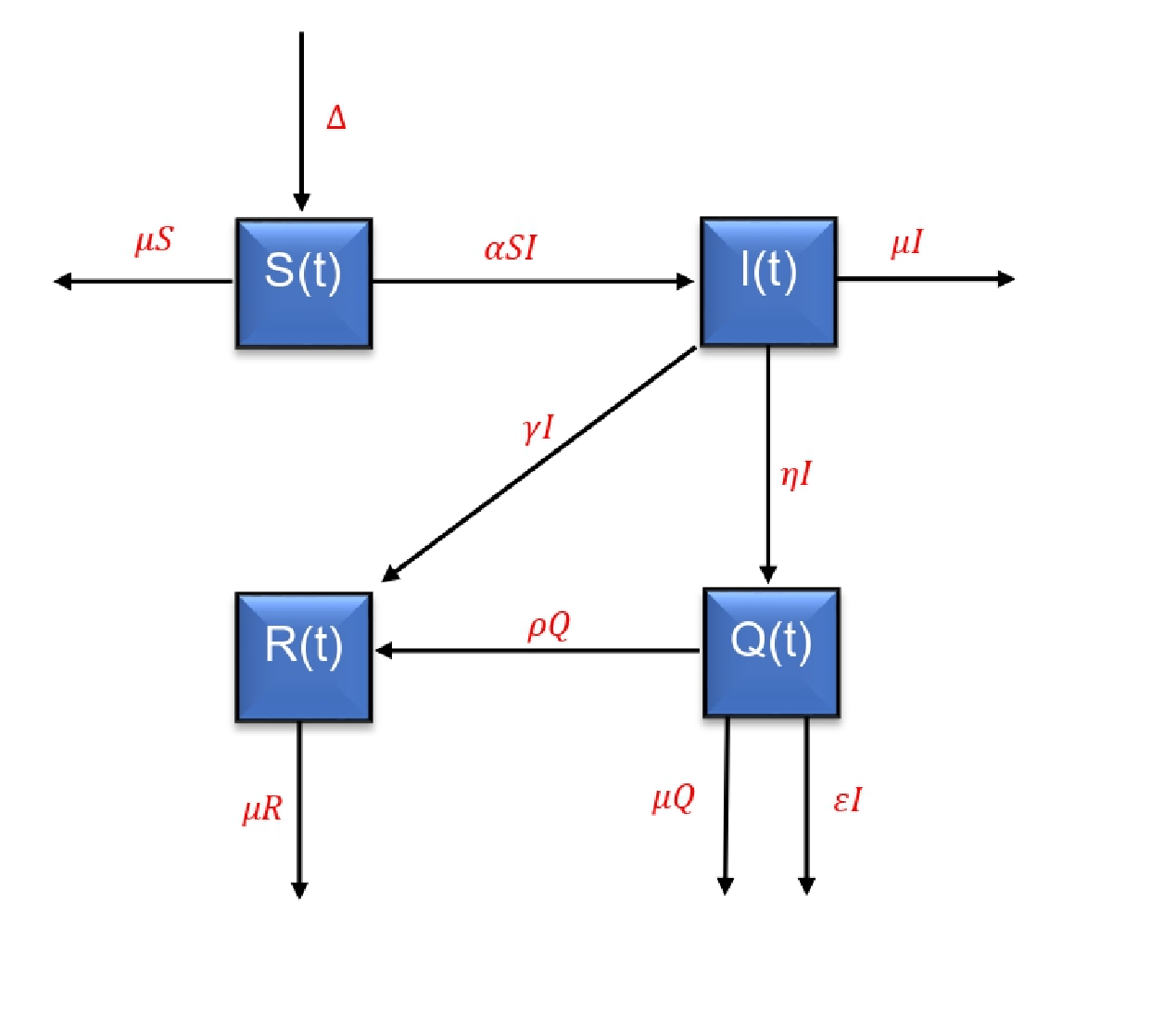}
	\caption{\label{fig:fluxo}SIQR's dynamic flow.}
\end{figure}

Some important observations about this model:

\begin{itemize}
	\item As already said, this is a homogeneous-mixing model;
	\item It uses only one disease-related death rate (some models have different death rates for infected and quarantine individuals, \cite{Leonardo} for example);
	\item This model does not consider an incubation period;
	\item It does not differentiate infected individuals with symptoms and without symptoms, they are all in the same group;
	\item In the quarantine compartment the individuals isolated are those who are infected, it does not consider the isolation of the healthy population.
\end{itemize}

\bigskip

More details about the properties of SIQR model can be found in \cite{aims}, \cite{Ma_global} and \cite{Elsevier}. In \cite{Leonardo} they use a slightly different SIQR model and even make simulations based on empirical data from local health authorities.

Now we complete our model by adding vaccination as a control agent for the system:

\begin{equation}\label{sist_vac}
	\begin{cases}
		\frac{dS}{dt} = \Delta - \alpha SI - \mu S - vS\\
		\frac{dI}{dt} = \alpha SI - (\gamma + \mu + \eta)I\\
		\frac{dQ}{dt} = (\eta - \epsilon)I - (\rho + \mu)Q\\
		\frac{dR}{dt} = \gamma I + \rho Q - \mu R
	\end{cases}\,.
\end{equation}

\bigskip

In system (\ref{sist_vac}) the $v$ parameter represents the presence of vaccination, the portion of the population who are vaccinated is subtracted from the susceptible group and added directly to the recovered group. It is easy to observe that the higher the percentage of vaccinated individuals, the lower the proportion of infected people and, consequently, the lower the number of deaths related to the disease. We now proceed to a qualitative analysis of the system of differential equations presented in (\ref{sist_vac}).

\Section{Metodology}

To study and analyze the stability of the disease-free equilibrium and endemic equilibrium points, we use the Routh-Hurwitz Criteria, the Lyapunov method and the La Salle Invariance Principle. 
To study and analyze the controllability of the control system, we use the Kalman criteria, we study the control problem with linear dynamics and quadratic objective function, we use the Ricatti equation to obtain the corresponding optimal controls. To solve the problems we use Runge-Kutta fourth order method and in numerical simulations we use MATLAB\textsuperscript{\textregistered} software.

\section{Equilibrium Points and Basic Reproductive Number}

In this section, we find the constant solutions (equilibrium points) of system (\ref{sist_vac}) and show the formula for the basic reproductive number.

\bigskip

We will call $R_{0}$ the basic reproductive number of the disease, which is an estimate of the contagiousness of the disease. This number is very important to health authorities and governments to devise strategies for facing the disease. To learn more about $R_{0}$ you can look at \cite{rep_numb} and \cite{keai}.
\begin{theorem} \label{theorem 3.1}
	The closed region $\Omega =\left\{  (S,I,Q,R) \in \mathbb{R}^{4}_{+}: N(t)\leq \frac{\Delta}{\mu}\right\}$ is positive invariant ste for the model (\ref{sist_vac}).
\end{theorem}
\begin{proof}: To initiate our analysis we have to establish boundaries to our variables and parameters. All parameters described in table (\ref{tab_param})  are real non-negative numbers. For our variables we Take the derivative of $N(t)$ in expression (\ref{pop}):
	
	$$N'(t) = \Delta - \mu N - \epsilon I \leq \Delta - \mu N$$
	
	\bigskip
	
	Multipling by integration factor $e^{\mu t}$:
	
	$$N'(t) e^{\mu t} \leq \Delta e^{\mu t} - \mu N e^{\mu t}$$
	
	$$N'(t) e^{\mu t} \leq \Delta e^{\mu t} - \mu N e^{\mu t}$$
	
	Integrating both sides:
	
	$$\int_{o}^{t} [N'(t) e^{\mu t} + \mu N e^{\mu t}] \,dt \leq \int_{o}^{t} \Delta e^{\mu t} \,dt $$
	
	$$N(t) e^{\mu t} \biggr |_{0}^{t} \leq \Delta \frac{e^{\mu t} - 1}{\mu}\biggr |_{0}^{t} $$
	
	$$ e^{\mu t} N(t) - N(0) \leq \Delta \biggr (\frac{e^{\mu t} - 1}{\mu}\biggr)$$
	
	$$N(t) \leq N(0) e^{-\mu t}  + \frac{\Delta}{\mu} (1 - e^{-\mu t})$$
	
	When $x\rightarrow \infty$ we have:
	
	$$\lim_{x\rightarrow \infty} Sup[N(t)] = \frac{\Delta}{\mu}$$
	
	\bigskip
	
	This shows us that we have to study the problem in the region $D$:
	
	$$D = \left\{ (S,I,Q,R)| S \geq 0, I \geq 0, Q \geq 0, R \geq 0, S + I + Q + R \leq \frac{\Delta}{\mu} \right\}$$
	
\end{proof}
\subsection{Positivity and Boundedness}

\begin{theorem}
	Let $(S(0), I(0), Q(0), R(0))$ be non negative initial conditions, then the solutions $(S(t), I(t), Q(t), R(t))$ of the proposed model $(\ref{sist_vac})$ are positive for all $t>0.$
\end{theorem}
\begin{proof} Consider the following from the model's first equation (\ref{sist_vac})
	\begin{eqnarray}\label{eq1} \nonumber
		\frac{dS}{dt}&=&\Delta - \alpha SI - \mu S - vS\\ \nonumber
		\frac{dS}{dt}     &\geq& - (\alpha I + \mu  + v)S\\ \nonumber
		S(t)&\geq& C e^{-(\alpha I+\mu+v)t}
	\end{eqnarray}
	
	Where $C = e^{C_1} $ is a constant determined by the initial conditions. Now, if $ S(0) $ is the initial condition, then $ S(0)= C $, so we can state that:
	\[ S(t) \geq S(0)e^{- (\alpha I + \mu + v)t} \]
	Therefore, the solution $S(t)$ is bounded above by $ S(0) $.\\
	Use the same justification, we can demonstrate that:
	\[Q(t)\geq Q(0)e^{-(\rho + \mu)t}\] and \[R(t)\geq R(0)e^{- \mu t}, \mbox{ for all $t>0$}.\]
	Now, we proof the lemma technical for show that $I(t)$ is bounded, in fact:
	\begin{lemma}
		Let the solution $S(t)$ be bounded below by $S(0)e^{-(\alpha I(t) + \mu + v)t}$ for all $t > 0$, where $S(0)$ is the initial condition. Then, the solution $I(t)$ of the differential equation
		\begin{equation*}
			\frac{dI}{dt} = \alpha SI - (\gamma + \mu + \eta)I
		\end{equation*}
		is bounded for all $t > 0$.
	\end{lemma}
	\begin{proof}
		We are given the constraint that $S(t) \geq S(0)e^{-(\alpha I(t) + \mu + v)t}$. This constraint means that the value of $S(t)$ will never decrease below \\
		$S(0)e^{-(\alpha I(t) + \mu + v)t}$. Under this condition, any increase in $I(t)$ that would cause $S(t)$ to decrease below the specified limit contradicts the imposed constraint on the dynamics of $S(t)$. This is a direct consequence of the interdependence between $S(t)$ and $I(t)$, as described by the equations (\ref{sist_vac})\\
		Therefore, the solution $I(t)$ is bounded, ensuring that $S(t)$ remains above the specified limit for all $t > 0$.
	\end{proof}
	This demonstrates that system (\ref{sist_vac}) solution is positive for all $t>0$. As a result suggested model epidemiologically significant and mathematically we posed in the domain $\Omega.$
\end{proof}
Set the right side of system (\ref{sist_vac}) to zero and we have:

\begin{equation}\label{sist_zero}
	\begin{cases}
		\Delta - \alpha SI - \mu S - vS = 0\\
		\alpha SI - (\gamma + \mu + \eta)I = 0\\
		(\eta - \epsilon)I - (\rho + \mu)Q = 0\\
		\gamma I + \rho Q - \mu R = 0
	\end{cases}\,.
\end{equation}

\bigskip

Now we establish the existence of a disease-free constant solution and a endemic constant solution:

\begin{theorem}
	For system (\ref{sist_zero}), there is always a disease-free equilibrium $E_{0}$, and there is also an unique endemic equilibrium $E^{*}$.
\end{theorem}

\begin{proof}[Proof]
	Observing the second equation in system (\ref{sist_zero}) we have the product:
	
	\begin{equation}
		I[\alpha S - (\gamma + \mu + \eta)] = 0
	\end{equation}
	
	\bigskip
	
	If we have $I = 0 \Rightarrow Q = 0$, $S = \frac{\Delta}{\mu + v}$ and $R = 0$
	
	Thus, the point $E_{0} = \biggr( \frac{\Delta}{\mu + v},0,0,0\biggr)$ is a solution called free-disease solution.
	
	\bigskip
	
	If $[\alpha S - (\gamma + \mu + \eta)] = 0$ then:
	
	$$S^{*} = \frac{\gamma + \mu + \eta}{\alpha}$$
	
	Replacing $S^{*}$ in the first equation of (\ref{sist_zero}) we get:
	
	$$\Delta - (\gamma + \mu + \eta) \biggr(I^{*} + \frac{\mu + v}{\alpha}\biggr) = 0$$
	
	$$I^{*} = \frac{\Delta}{\gamma + \mu + \eta} - \frac{\mu + v}{\alpha}$$
	
	$$I^{*} = \frac{\mu + v}{\alpha} \biggr( \frac{\Delta}{\mu + v} \frac{\alpha}{\gamma + \mu + \eta} - 1 \biggr)$$
	
	Proceeding in the same way in the third and fourth equations, we obtain:
	
	$$Q^{*} = \frac{(\eta - \epsilon) I^{*}}{\rho + \mu}$$
	
	$$R^{*} = \frac{ \gamma I^{*} + \rho Q^{*}}{\mu}$$
	
	$E^{*} = (S^{*},I^{*},Q^{*},R^{*})$ is called the endemic solution.
\end{proof}

Looking at the expression of $I^{*}$ it is important to note that the equilibrium point $E^{*}$ only occurs if: 

$$\biggr( \frac{\Delta}{\mu + v} \frac{\alpha }{\gamma + \mu + \eta} - 1 \biggr) > 1$$

So, we define:
\begin{definition}
	The Basic reproduction number for sistem $(\ref{sist_vac})$, denoted by $R_{0}$, is given as:
	\begin{equation}\label{r0}
		R_{0} = \biggr(\frac{\Delta}{\mu + v} \frac{\alpha}{\gamma + \mu + \eta} \biggr).
	\end{equation}
	which represents the average number of new infections generated by an infectious case in a susceptible population.
\end{definition}

The expression (\ref{r0}) is essential to system (\ref{sist_vac}), if $R_{0} > 1$ then the solution converge to the endemic equilibrium, on the other hand, if $R_{0} < 1$ then the solution converge to the free-disease equilibrium. 

The value of $R_{0}$ for COVID-19 is estimated to be between $1.9$ and $6.5$ according to \cite{rep_numb}.

\section{Stability of Equilibrium Points}

In this section we will state and prove the theorems that establish the stability of the solutions found in the previous section.

\begin{theorem}\label{theorem 4.1}
	If $R_{0} < 1$, the disease-free equilibrium $E_{0}$ of the system (\ref{sist_vac}) is locally asymptotically stable. If $R_{0} > 1$,the disease-free equilibrium $E_{0}$ is unstable.
\end{theorem}

\begin{proof}[Proof]
	
	The Jacobian matrix of system $(\ref{sist_vac})$ at $E_{0}$ is:
	
	$$J(E_{0}) = \left[\begin{array}{cccc} - \mu - v & -\frac{\alpha \Delta}{\mu + v}  & 0 & 0 \\ 0 & \frac{\alpha \Delta}{\mu + v}  - (\gamma + \mu + \eta) & 0 & 0 \\ 0 & (\eta - \epsilon) & -(\rho + \mu) & 0 \\ 0 & \gamma & \rho & - \mu \end{array}\right]$$
	
	The four eigenvalues of matrix $J(E_{0})$ are:
	
	$$\lambda_{1} = - \mu - v, \lambda_{2} = \frac{\alpha}{\gamma + \mu + \eta} (R_{0} - 1), \lambda_{3} = -(\rho + \mu), \lambda_{4} = -\mu$$
	
	If $R_{0} < 1 \Rightarrow \lambda_{2} < 0$, therefore, all eigenvalues have negative real parts and $E_{0}$ is locally asymptotically stable. If $R_{0} > 1 \Rightarrow \lambda_{2} >0$, thus,  $E_{0}$ is unstable.
\end{proof}

\bigskip

\begin{theorem}\label{theorem 4.2}
	If $R_{0} < 1$, the disease-free equilibrium $E_{0}$ of the system $(\ref{sist_vac})$ is globally asymptotically stable.
\end{theorem}

\begin{proof}[Proof]
	
	Consider the following Lyapunov function:
	
	$$\mathcal{L}(t) = I(t)$$
	
	Calculating the derivative of $\mathcal{L}(t)$ along the positive solution of system $(\ref{sist_vac})$, it follows that:
	\begin{eqnarray}
		\frac{d\mathcal{L}(t)}{dt}\biggr|_{(3)} &=& \frac{dI}{dt}\biggr|_{(3)} = \alpha SI - (\gamma + \mu + \eta)I\\
		&=&  [\alpha S - (\gamma + \mu + \eta)]I\\
		&=& \biggr[ \alpha \frac{\Delta}{\mu + v} - (\gamma + \mu + \eta) \biggr] I\\
		&=& [(\gamma + \mu + \eta)(R_{0} - 1)]I \\
		&\leq &0
	\end{eqnarray}
	\bigskip
	
	Furthermore, $\mathcal{L}'=0$ only if $I=0$. The maximum invariant set in $\left\{(S,I,Q,R)|\mathcal{L}' = 0\right\}$ is the singleton ${E_{0}}$. When $R_{0} < 1$, according to LaSalle's invariance principle, it follows that:
	
	\bigskip
	
	\[ \lim_{t\to\infty} I(t) = 0\]
	
	\bigskip
	
	Then, we obtain the limit equations of the system $(\ref{sist_vac})$:
	
	$$
	\begin{cases}
		\frac{dS}{dt} = \Delta - \mu S - vS\\
		\frac{dQ}{dt} = - (\rho + \mu)Q\\
		\frac{dR}{dt} =  \rho Q - \mu R
	\end{cases}\,.
	$$
	
	\bigskip
	
	So the disease-free equilibrium is globally attractive in the region D. Therefore, the disease-free equilibrium of system $(\ref{sist_vac})$ is globally asymptotically stable when $R_{0} < 0$.
\end{proof}

\bigskip

\begin{theorem}\label{theorem 4.3}
	If $R_{0} > 1$, the endemic equilibrium $E^{*}$ of the system (\ref{sist_vac}) is locally asymptotically stable.
\end{theorem}

\begin{proof}[Proof]
	
	The Jacobian matrix of system $(\ref{sist_vac})$ at $E^{*}$ is:
	
	$$J(E^{*}) = \left[\begin{array}{cccc} - \alpha I^{*} - \mu - v & - \alpha S^{*} & 0 & 0 \\ \alpha S^{*} & 0 & 0 & 0 \\ 0 & (\eta - \epsilon) & -(\rho + \mu) & 0 \\ 0 & \gamma & \rho & - \mu \end{array}\right]$$
	
	\bigskip
	
	The two eigenvalues of matrix $J(E^{*})$ are:
	
	$$\lambda_{3} = - (\rho + \mu), \lambda_{4} = -\mu$$
	
	\bigskip
	
	The other two eigenvalues are also the eigenvalues of the following matrix:
	
	$$J^{*}(E^{*}) = \left[\begin{array}{cccc} - (\alpha I^{*} + \mu + v) & - \alpha S^{*} \\ \alpha S^{*} & 0 \end{array}\right]$$
	
	\bigskip
	
	which has the characteristic polynomial:
	
	$$\lambda^{2} + (\alpha I^{*} + \mu + v) \lambda + \alpha^{2} S^{*^2} = 0$$
	
	\bigskip
	
	which has all coefficients positive, applying the Routh-Hurwitz criterion we obtain that all eigenvalues of matrix $J(E^{*})$ have negative real parts and the endemic equilibrium $E^{*}$ is locally asymptotically stable.
\end{proof}

\begin{theorem}\label{theorem 4.4}
	If $R_{0} > 1$, the endemic equilibrium $E^{*}$ of the system (\ref{sist_vac}) is globally asymptotically stable.
\end{theorem}

\begin{proof} If $R_{0}>1$ we have that endemic equilibrium point values are given by  
	\begin{equation*}
		E^{*} = \biggr( \frac{\gamma + \mu + \eta}{\alpha},\frac{\mu + v}{\alpha}(R_{0} - 1 ),\frac{(\eta - \epsilon) I^{*}}{\rho + \mu},\frac{ \gamma I^{*} + \rho Q^{*}}{\mu}\biggr).
	\end{equation*}
	Consider the following Liapunov function
	\begin{equation*}
		\mathcal{L}(t) = \frac{1}{2}\left[(S-S^{*}) + (I- I^{*})+ (Q-Q^{*})+(R-R^{*})\right]^{2}. 
	\end{equation*}
	Calculating the derivative of $\mathcal{L}(t)$ with respect $t$, follows that
	\begin{equation*}
		\frac{d \mathcal{L}}{dt} = \left[(S-S^{*}) + (I- I^{*})+ (Q-Q^{*})+(R-R^{*})\right]	\frac{d}{dt}\left[S + I+ Q+R\right],
	\end{equation*}
	since that $N(t)=S+I+Q+R$ and by Theorem $\ref{theorem 3.1}$, we have $	N'(t)\leq (\Delta-\mu N)$ and consequently $N(t)\leq \frac{\Delta}{\mu}$. So, 
	\begin{eqnarray}\nonumber
		\frac{d \mathcal{L}}{dt} &=&\left[(S-S^{*}) + (I- I^{*})+ (Q-Q^{*})+(R-R^{*})\right]	\frac{d N}{dt}\\ \nonumber
		&\leq & \left[(S-S^{*}) + (I- I^{*})+ (Q-Q^{*})+(R-R^{*})\right](\Delta-\mu N)\\ \nonumber
		&\leq&(N-\frac{\Delta}{\mu})(\Delta-\mu N)
	\end{eqnarray}
	Finally we get, 
	
	\begin{equation}
		\frac{d \mathcal{L}}{dt} \leq -\frac{1}{\mu}(\Delta-\mu N)^{2}.
	\end{equation}
	We can clearly determine that $\mathcal{L}'(t)$ is negative definite and $\mathcal{L}(t)$ is positive definite. Furthermore, $\frac{d \mathcal{L}}{dt}=0 $
	if and only if $S=S^{*}, I=I^{*}, Q=Q^{*}, R=R^{*}$. Therefore, the largest compact invariant set of $\left\lbrace \mathcal{L}'(t)=0  \right\rbrace $ is the singleton $E^{*}$. This shows that by the classical Lyapunov and La Salle invariance principle, $E^{*}$ is globally asymptotically stable.
	Thus, the system $(\ref{sist_vac})$ has a globally asymptotically stable solution $(S^{*}, I^{*}, Q^{*}, R^{*})$.

\end{proof}

\section{Control of finite dimensional linear systems}
Let $T > 0$. We consider the following finite dimensional system:
\begin{eqnarray}\label{eq2}
	x'(t) &= &Ax(t)+Bu(t), \; t \in [0,T]\\ \nonumber
	x(0)&=&x_{0},
\end{eqnarray}
where $A \in \mathbb{R}^{n,n}$,  $B \in \mathbb{R}^{n,m}$ are a real matrix, and $x_{0}$ a vector in $\mathbb{R}^{n}$. The function $x:[0, T] \rightarrow \mathbb{R}^{n}$ represents the state and $u : [0, T]\rightarrow \mathbb{R}^{m}$
the control. Both are vector functions of $m$ and $n$ components respectively depending exclusively on time $t$.\\
Given an initial datum $x_{0} \in \mathbb{R}^{n}$  and a vector function  $u \in L^{1}([0,T];\mathbb{R}^{m})$   , system
$(\ref{eq2})$ has a unique solution $C([0,T];\mathbb{R}^{n})$ characterized by the variation of constants formula:
\begin{equation}\label{eq3}
	x(t)= e^{At}x_{0}+\int_{0}^{t}e^{A(t-s)}Bu(s)ds,\; \forall t \in [0,T].
\end{equation}

\subsection{Kalman's controllability condition}
The following classical result is due to \cite{micu} and gives a complete answer to the problem of exact controllability of finite dimensional linear systems.
It shows, in particular, that the time of control is irrelevant.

We considering $A \in \mathbb{R}^{4,4}$ is Jacobian matrix of system $(\ref{sist_bas})$ at $E_{0}$ without the perturbation of control,  $B \in \mathbb{R}^{4,2}$ is a real matrix, and $x_{0}$ a vector in $\mathbb{R}^{4}$. The function $x:[0, T] \rightarrow \mathbb{R}^{4}$ represents the state and $u : [0, T]\rightarrow \mathbb{R}^{2}$
the control. Both are vector functions of $4$ and $2$ components respectively depending exclusively on time $t$. Let us use the shorthand notation $(A, B)$ to denote the control system $(\ref{eq2}).$

\begin{definition}
	A system $(A,B)$, for which the Kalman criterion condition is satisfied, is termed completely controllable.
\end{definition}

\begin{theorem}
	System $(\ref{eq2})$ is completely  controllable in some time $T$ if and only if
	\begin{equation}\label{eq4}
		rank [B, AB, A^{2}B,A^{3}B] = 4.
	\end{equation}
	Consequently, if system $(\ref{eq2})$ is controllable in some time $T > 0$ it is controllable in any time.
\end{theorem}
\begin{proof}
	In fact, from $(\ref{eq2})$ we had 
	\[ A = \begin{bmatrix} -\mu & 0 & 0 & 0 \\ 0 & - (\gamma + \mu + \eta) & 0 & 0 \\ 0 & (\eta - \epsilon) & -(\rho + \mu) & 0 \\ 0 & \gamma & \rho & -\mu \end{bmatrix} \]
	and \[ B = \begin{bmatrix} 1 & 0  \\ 0 & 1  \\ 0 & 0 \\ 0 & 0 \end{bmatrix}. \] with $\mu$, $\gamma$, $\eta$, $\epsilon$ e $\rho$ are positive constants. To determine the controllability of the system, we calculate the controllability matrix $W_{c}=[B, AB, A^{2}B,A^{3}B].$ Consequently, we obtain
	\begin{equation}
		rank\; W_{c} = 4.
	\end{equation}
	Therefore, the system is controllable.
	In conclusion, the matrix calculation is shown in \textbf{Algorithm 1}.
	\begin{table}[h!]
		\begin{tabular}{l}
			\noindent\rule{10cm}{1.5pt}\\
			\text{\textbf{Algorithm 1}. Solve the matrix $W_{c}$} \\
			\noindent\rule{10cm}{0.8pt}\\
			1:  \textbf{Constant values}\\
			2: mu = 0.02;\\
			3: alpha = 0.2;\\
			4: Delta = 0.2;\\
			5: gamma = 0.1;\\
			6: eta = 0.2;\\
			7: epsilon = 0.1;\\
			8: rho = 0.3;\\
			9: \textbf{ Definition of matrix A}\\
			10: A = [-mu , 0, 0, 0; 0, -(gamma + mu + eta), 0, 0;\\
			11: 0, eta - epsilon, -(rho + mu), 0;\\
			12: 0, gamma, rho, -mu];\\
			13: \textbf{Definition of matrix B }\\
			14: B = [1 0;0 1;0 0;0 0];\\
			15: \textbf{Controllability check}\\
			16: $Wc = [B, AB, A^2B, A^3B]$;\\
			17: $rank_Wc = rank(Wc)$;\\
			18: \textbf{Show the position of the $W_c$ matrix}\\
			19: $disp(['Posto\; da\; matriz\; Wc:\; ' num2str(rank_Wc)])$\\
			\noindent\rule{10cm}{0.8pt}\\
		\end{tabular}
	\end{table}
\end{proof} 

In Figure \ref{fig:2}, we observe that the number of columns in Matrix \( W_c \) equals the order of the system, indicating complete controllability. However, columns 1 and 2 exhibit singular values significantly different from zero, suggesting their importance in system controllability. Thus, the presence of these significant singular values implies that columns 1 and 2 of the controllability matrix are crucial for system controllability.

\newpage
\begin{figure}[ht]
	\centering
	\includegraphics[width=0.8\textwidth]{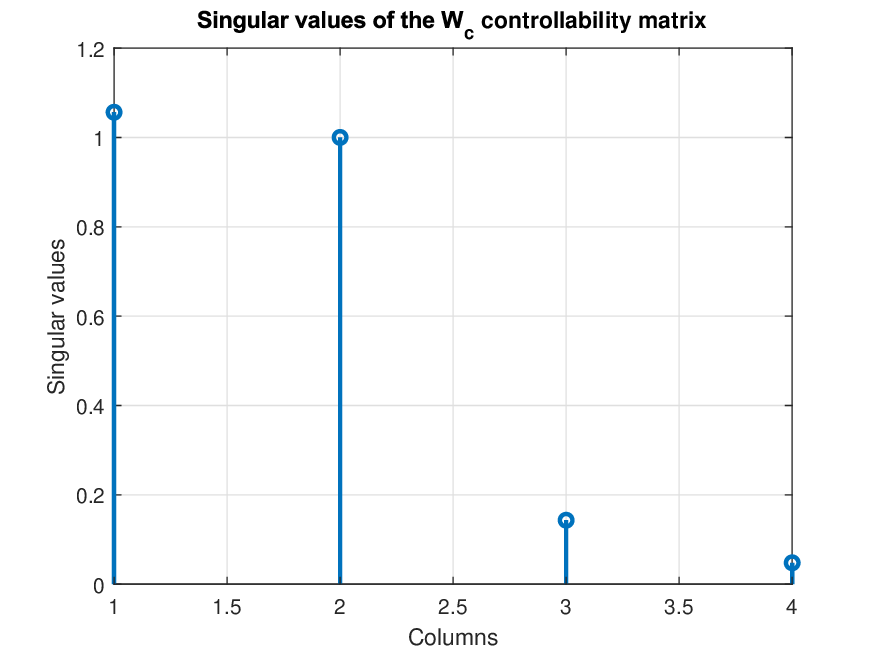}
	\caption{\label{fig:2} Controllabillity Matrix}
\end{figure}

In the context of working with an epidemiological model that provides insights into the spread of Covid-19, it is highly advantageous to employ output feedback and State Feedback Strategies to modify the dynamics of the free system, aiming to achieve properties such as full controllability, asymptotic stability, BIBO-stability, etc. At this stage of the research, we are interested in demonstrating that in the case of the autonomous linear control system. Given $A \in \mathbb{R}^{n,n}$ and $B \in \mathbb{B}^{n,m}$, we have
\begin{equation}\label{eq5}
	x'=Ax+Bu.
\end{equation}
We can utilize the state feedback strategy, where we assume that the control $u$ is derived from the state $x$ through a linear law, denoted as 
\begin{equation}\label{eq6}
	u=-Fx,
\end{equation}
where $F \in \mathbb{R}^{m,n}$ is the state feedback \textit{Gain Matrix}. Substituting them $(\ref{eq5})$, we get
\begin{equation}\label{eq7}
	x'=(A+BF)x.
\end{equation}

The subsequent outcome facilitates our examination of controllability concerns in the presence of disturbances within a completely controllable autonomous system.
\begin{theorem}\label{theorem 5.2}
	Let $(A,B)$ be a completely controllable autonomous system. Then, for every matrix $F \in \mathbb{R}^{2,4}$, the system $(A+BF, B)$ is also completely controllable.
\end{theorem}
\begin{proof}
	Note that \[ A = \begin{bmatrix} -\mu & 0 & 0 & 0 \\ 0 & 0 - (\gamma + \mu + \eta) & 0 & 0 \\ 0 & (\eta - \epsilon) & -(\rho + \mu) & 0 \\ 0 & \gamma & \rho & -\mu \end{bmatrix} \]
	and
	\[ B = \begin{bmatrix} 1 & 0  \\ 0 & 1  \\ 0 & 0 \\ 0 & 0 \end{bmatrix}. \]
	Given 
	\[F =-\begin{bmatrix} v& \frac{\alpha \Delta}{\mu+v}  &0&0\\ 0 &  -\frac{\alpha \Delta}{\mu+v}&0 &0 \end{bmatrix}\]
	by $(\ref{eq6})$ assume that the control $u$ is
	\[ u = -\begin{bmatrix} v& \frac{\alpha \Delta}{\mu+v}  &0&0\\ 0 &  -\frac{\alpha \Delta}{\mu+v}&0 &0 \end{bmatrix} \begin{bmatrix} S\\ I \\ Q \\ R\end{bmatrix}, \]
	that result
	\[ u = \begin{bmatrix} -vS - \frac{\alpha \Delta}{\mu+v}I  \\  \frac{\alpha \Delta}{\mu+v}I \end{bmatrix}. \]
	Consequently, we obtain
	\[ (A+BF) = \begin{bmatrix} -\mu-v  & -\frac{\alpha \Delta}{\mu+v} & 0 & 0 \\ 0 & \frac{\alpha \Delta}{\mu +v } - (\gamma + \mu + \eta) & 0 & 0 \\ 0 & (\eta - \epsilon) & -(\rho + \mu) & 0 \\ 0 & \gamma & \rho & -\mu \end{bmatrix}, \]
	note that $(A+BF)$ is the Jacobian matrix of the system $(\ref{sist_vac})$ at point $E_{0}$ with the presence of the control perturbation. To prove the given theorem, we will use the Kalman criterion for the controllability of linear systems.
	The Kalman criterion states that a linear system is completely controllable if its controllability matrix:
	\[ \mathcal{C} = [B \quad AB \quad A^2B \quad A^3B] \]
	has full rank, meaning its rank equals the dimension of the state space.
	Now, let's use the Kalman criterion to prove the given theorem.
	Given a system \( (A, B) \) that is completely controllable, we know that the controllability matrix \( \mathcal{C} \) has full rank, which means the rank of \( \mathcal{C} \) equals the dimension of the state space.
	Let the dimension of the state space equal \( 4 \). Then, the rank of \( \mathcal{C} \) is \( 4 \).
	Now, consider the system \( (A + BF, B) \), where \( F \) is a control matrix.
	The controllability matrix of this system is:
	\[ \mathcal{C}_{\text{new}} = [B \quad (A+BF)B \quad (A+BF)^2B \quad (A+BF)^3B] \]
	Now, we need to show that the rank of \( \mathcal{C}_{\text{new}} \) is equal to \( 4 \), i.e., full rank. If we can show this, then the system \( (A + BF, B) \) will be completely controllable.
	To do this, we will use the property that the rank of a matrix is not changed when we multiply the matrix by another invertible matrix on the left. We will use this property to show that the rank of \( \mathcal{C}_{\text{new}} \) is equal to \( 4 \).
	Consider the matrix \( T \) defined as:
	\[ T = \begin{bmatrix} I & 0 & 0 & 0 \\ F & I & 0 & 0 \\ F^2 & 2BF & I & 0 \\ F^3 & 3BF^2 & 3B(F^2+BF+I) & I \end{bmatrix} \]
	Where \( I \) is the identity matrix. It is easy to see that \( T \) is an invertible matrix. 
	Furthermore, if we multiply \( T \) by the matrix \( \mathcal{C}_{\text{new}} \), we obtain \( \mathcal{C} \).
	\[ T\mathcal{C}_{\text{new}} = [B \quad AB \quad A^2B \quad A^3B] = \mathcal{C} \]
	Since \( \mathcal{C} \) has full rank (equal to \( n \)), then \( T\mathcal{C}_{\text{new}} \) also has full rank.
	Therefore, the rank of \( \mathcal{C}_{\text{new}} \) is equal to \( 4 \), and thus the system \( (A + BF, B) \) is completely controllable.
	Thus, using the Kalman criterion, we have proven that for every matrix \( F \), the system \( (A + BF, B) \) is completely controllable, provided that \( (A, B) \) is completely controllable. In conclusion, the calculation of the matrix \((A+BF)\) is shown in \textbf{Algorithm 1} with the appropriate changes to the input values. This concludes the proof of the theorem. 
\end{proof}

\section{Optimal Control Model}

In this section, we associate the control problem $(\ref{sist_vac})$ with a function that is intrinsically related to solving the system problem. This relationship is described through the optimality principle of a dynamic system. We want to find a control function that minimizes or maximizes a cost functional, while satisfying the constraints imposed by the system. Thus, the following optimal control variable is given: The variable $u(t)$ represents vaccination, as seen previously. In this context, optimal control theory provides a powerful framework for designing control strategies that minimize the spread of infectious diseases while considering various constraints and objectives. By formulating the problem as an optimization task, optimal control theory allows us to determine the most effective allocation of control measures over time to achieve specific objectives, such as minimizing the number of infections, reducing economic losses, or optimizing the use of healthcare resources.

We treat a special case in the optimal control of systems, in
which the state diferential equations are linear in $x$ and $u$ and the objective
functional is quadratic.

Let $T>0$ fixed. Given $t_{0} \in [0,T]$ and $x_{0} \in \mathbb{R}^{n}$, considering
a dynamic system described by the following differential equations:
\begin{eqnarray}\label{eq8}
	x'(t) &= &Ax(t)+Bu(t), \; t \in [0,T]\\ \nonumber
	x(t_{0})&=&x_{0},
\end{eqnarray}
consider the cost functional
\begin{equation}\label{eq9}
	J(u,x) = \frac{1}{2}\left[\int_{0}^{T}x^\mathbf{T}(t)G(t)x(t) + u^\mathbf{T}(t)R(t)u(t)dt\right]  
\end{equation}
where the  matrices $G(t)$ and $R(t)$ are sizes $n\times n $, $m\times m$ respectively, $G(t)$ being positive semidefinite and $R(t)$ being positive definite for all $t \in [0,T]$. The positive definite property guarantees $R(t)$ is invertible. The superscript $\mathbf{T}$ refers to transpose of the matrix. The set of admissible control. 
\begin{equation}\label{eq10}
	\mathcal{U}_{ad}:=\left\{u \in L^{1}([0,T];\mathbb{R}^{m}); u(t) \in  \mathcal{X}\subset{\mathbb{R}^{m}}\; a.e\; in \;  [t_{0},T]. \right\}
\end{equation}
\begin{definition}
	The optimal value function is the application $V:[0,T] \times \mathbb{R}^{n} \rightarrow \mathbb{R}$ defined by 
	\begin{equation}
		V(t_{0},x_{0}):= inf\left\{J(t_{0},x_{0}; u); \; u \in \mathcal{U}_{ad}\right\}.		
	\end{equation}
\end{definition}

\begin{definition}
	The Optimal Control problem is to find for a given initial condition
	$x_{0}$, the control $u^{\ast} \in \mathcal{U}_{ad}$ that  minimizes the cost fucntional $(\ref{eq9})$. Furthermore,
	\[V(t_{0},x_{0})=J(t_{0},x_{0}; u^{\ast}).\]
\end{definition}

With the objective of illustrating the ideas presented here, we consider the control problem $(\ref{eq8})$ to $(\ref{eq9})$. The Hamiltonian becomes
\begin{equation}\label{eq11}
	H(t,x,u,\lambda) :=\frac{1}{2}x^\mathbf{T}(t) Gx(t) + \frac{1}{2}u^\mathbf{T} R u + \lambda^T(Ax(t) + Bu).
\end{equation} 
From of the Hamiltonian we have the optimality equation
is Derived from the term \( u^\mathbf{T} R u \) with respect to \( u \):
\[ \frac{\partial}{\partial u} \left( u^\mathbf{T} R u \right) = \frac{\partial}{\partial u} \left( \sum_{i=1}^{n} \sum_{j=1}^{n} u_i R_{ij} u_j \right) \]
\[ = \frac{\partial}{\partial u} \left( \sum_{i=1}^{n} \sum_{j=1}^{n} u_i \frac{1}{2}(R_{ij} u_j + R_{ji} u_i) \right) \]
\[ = \frac{1}{2} \left( R + R^\mathbf{T} \right) u \]
\[ = R u \]

Derived from the term \( \lambda^\mathbf{T}(Ax(t) + Bu) \) with respect to \( u \):
\[ \frac{\partial}{\partial u} \left( \lambda^\mathbf{T}(Ax(t) + Bu) \right) = \frac{\partial}{\partial u} \left( \lambda^\mathbf{T} Bu \right) \]
\[ = B^\mathbf{T} \lambda \]

Therefore, the partial derivative of \( H \) ith respect to \( u \) is:
\[ \frac{\partial H}{\partial u} = R u + B^\mathbf{T} \lambda,\]
whence it follows that
\begin{equation}\label{eq12}
	u^{\ast}=-R^{-1}B^{\mathbf{T}}\lambda
\end{equation}
we have that
\begin{equation}
	\mathcal{H}(t,x, \lambda) = \stackrel{{\Large min}}{{\scriptsize u \in \mathcal{X}}} \left\lbrace \left\langle \lambda^{\mathbf{T}},Ax+Bu\right\rangle  + \frac{1}{2}\left[x^\mathbf{T}Gx + u^\mathbf{T}Ru\right] \right\rbrace,
\end{equation}
then using the \textit{Hamilton-Jokobi-Bellman} optimality equation
\begin{eqnarray}\label{eq13}
	\frac{\partial V}{\partial t}+ \left\langle \frac{\partial V}{\partial x},Ax\right\rangle- \frac{1}{2} \left\langle\frac{\partial V}{\partial x},BR^{-1}B^{\mathbf{T}}\frac{\partial V}{\partial x} \right\rangle + \frac{1}{2}\left\langle x^\mathbf{T},Gx\right\rangle=0
\end{eqnarray}
Let us now make the most important development hypothesis, which allows us to determine $V$. Suppose the value function for the linear quadratic problem has the form:
\begin{equation}\label{eq14}
	V(t,x):=\frac{1}{2}\left\langle x, P(t)x \right\rangle, \; (t,x) \in [0,T]\times \mathbb{R}^{n},
\end{equation}
where we assume that $P: [0,T]\rightarrow \mathbb{R}^{n,n}$ is continuously differentiable, as the cost function is non-decreasing, we can affirm that $P(t)$ it is positive defined for all $t \in [0,T]$. The assumptions of symmetry for $G$, $R$ are buried in the above calculations. Instead of using $\lambda$, wen find a matrix function P(t) such that $\lambda(t) =\frac{\partial V}{\partial x}= P(t)x(t)$.
Substituting the expression for $V$ into the \textit{HJB} equation, we obtain
\begin{equation}\label{eq15}
	\left\langle x, Y(t)x\right\rangle =0,
\end{equation} 
where
\begin{equation}\label{eq16}
	Y(t)= P'(t)+P(t)A+A^{\mathbf{T}}P(t)-P(t)BR^{-1}B^{\mathbf{T}}P(t)+G.
\end{equation}
the problem now is to find a matrix function $P$ such that $Y(t)\equiv 0$. The following theorem guarantees this result.
\begin{theorem}\label{theorem 7.1}
	Let \( P(t) \) be a continuous and differentiable symmetric matrix with respect to time \( t \) on an interval \( [0, T] \). Consider the Riccati equation: 
	\begin{equation}\label{eq17}
		\frac{\partial P(t)}{\partial t} = -A^T P(t) - P(t)A + P(t)B R^{-1} B^T P(t) - G
	\end{equation}
	where \( A \) is a constant matrix of size \( n \times n \), \( B \) is a constant matrix of size \( n \times m \), \( R \) is a positive definite  matrix of size \( m \times m \), and \( G \) is a constant symmetric matrix of size \( n \times n \), where the optimal control is of the form
	\begin{equation}\label{eq18}
		u^{\ast}=-R^{-1}B^{\mathbf{T}}\lambda.
	\end{equation}
	Then, for each initial condition \( P(0) = P_0 \), there exists a unique solution \( P(t) \) to the Riccati equation defined on \( [0, T] \) associated with the control system $(\ref{eq8})$ to $(\ref{eq9})$.
\end{theorem}

Simple ODE techniques can be used to solve the problem because, once the \textit{Riccati} matrix equation for $P$ is solved, the control is given by an equation in $x$, and $x$ is given by an ODE in $u$. The proof for this theorem can be seen in detail in \cite{baumeister}.

Now, we consider the control system (\ref{eq8}) whose governing matrix
\[ A = \begin{bmatrix} -\mu & 0 & 0 & 0 \\ 0 & - (\gamma + \mu + \eta) & 0 & 0 \\ 0 & (\eta - \epsilon) & -(\rho + \mu) & 0 \\ 0 & \gamma & \rho & -\mu \end{bmatrix} \]
The control operator is assumed to be of form
\[ B = \begin{bmatrix} 1 & 0  \\ 0 & 1  \\ 0 & 0 \\ 0 & 0 \end{bmatrix}, \]
and
\begin{minipage}{0.4\textwidth}
	\[ R = \begin{bmatrix} 2 & 0  \\ 0 & 2   \end{bmatrix}, \]
\end{minipage}
\begin{minipage}{0.3\textwidth}
	\[ G = \begin{bmatrix} 
		1 & 0 & 0 & 0 \\ 
		0 & 1 & 0 & 0 \\
		0 & 0& 0 & 0 \\ 
		0 & 0& 0 & 0\end{bmatrix}. \]
\end{minipage}

The $(\ref{eq8})$ system is added to a control measure to decrease the transmission of COVID-19, in which the optimal control policy
$u^{\ast}=-R^{-1}B^{\mathbf{T}}\lambda$ determine how the system should be controlled to minimize the cost function. The main objective of optimal control is to reduce the number of individuals susceptible $S(t)$ to and infected $I(t)$ with COVID-19 in the population, while also reducing the overall cost of controlling the disease dynamics. Then, the cost functional $(\ref{eq9})$ can be rewritten as:
\begin{equation}
	J(u,x) = \frac{1}{2}\left[\int_{0}^{T}(S^{2}+I^{2} + u^{2})dt\right],
\end{equation}
with, $u^{*}=-[1/2\lambda \; \; 1/2\lambda]^{\mathbf{T}}.$

Solving the Riccati equation for internal control $u^{\ast}(t)$:
\begin{equation}\label{eq.riccati}
	\frac{\partial P(t)}{\partial t} = -A^T P(t) - P(t)A + P(t)B R^{-1} B^T P(t) - G.
\end{equation}

For this example we specify the following ingredients: $T=30$;
$ \mu = 0.02$; $\alpha = 0.2$; $\Delta = 0.2$; $\gamma = 0.1$; $\eta = 0.2$; $\epsilon = 0.1$; $\rho = 0.3$; $\lambda = 1.5 $. And using the Runge-Kutta fourth order method we solve the equation, see figure $(\ref{fig:3})$.

\begin{figure}[ht]
	\centering
	\includegraphics[width=0.8\textwidth]{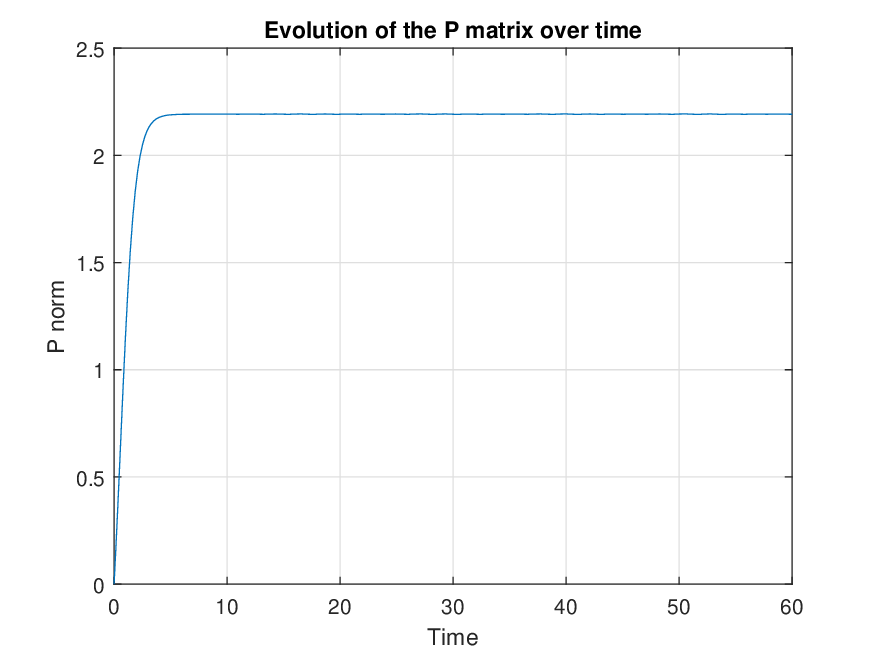}
	\caption{\label{fig:3} Evolution of the P norm}
\end{figure}
The \(P(t)\) matrix, obtained as the solution of the Riccati equation, is directly related to optimal control. This tool is used to find the control input that minimizes the cost of the system over time. The \(P(t)\) matrix norm in the context of the Riccati equation provides information about the stability of the controlled system, indicating whether or not it has stabilized as it converges to a constant value over time. The graph shows, see figure $(\ref{fig:3})$, this evolution by indicating how optimal control also converges and ensures both stability and adequate performance for the entire controlled system. In the next section we will show the optimal control numerically.

\section{Numerical Simulations}

We will see numerical simulations in this section, to demonstrate the dynamic characteristics of the model, including the stability of the equilibrium points and optimal control. Using the MATLAB\textsuperscript{\textregistered} software, we perform numerical simulation in the model $(\ref{sist_vac})$ and estimate the basic values of the model parameters. We will see the results of stabilization to the endemic and disease-free equilibrium points. We show how to solve the suggested optimal control problem. Let us simulate and compare different situations to control the spread of COVID-19. The results of the simulation are shown in the following diagram.\\

In system (\ref{sist_vac}), let $\gamma = 0.1$, $\mu = 0.02$, $\rho = 0.3$, $\epsilon = 0.1$, $\eta = 0.2$, $\Delta = 0.2$, and $v = 0.05$. When $\alpha=0.08 $ we have $R_{0} = 0.7143 < 1$, and with the initial condition $(9,1,0,0)$ the solution converge to the free-disease solution $(2.8571,0,0,0.0041)$. The numerical simulation is shown in the figure $\ref{fig:4}$. From Theorem \ref{theorem 4.2}, we notice that $E_{0}$ is globally asymptotically stable.
\begin{figure}[h!]
	\centering
	\includegraphics[width=0.8\textwidth]{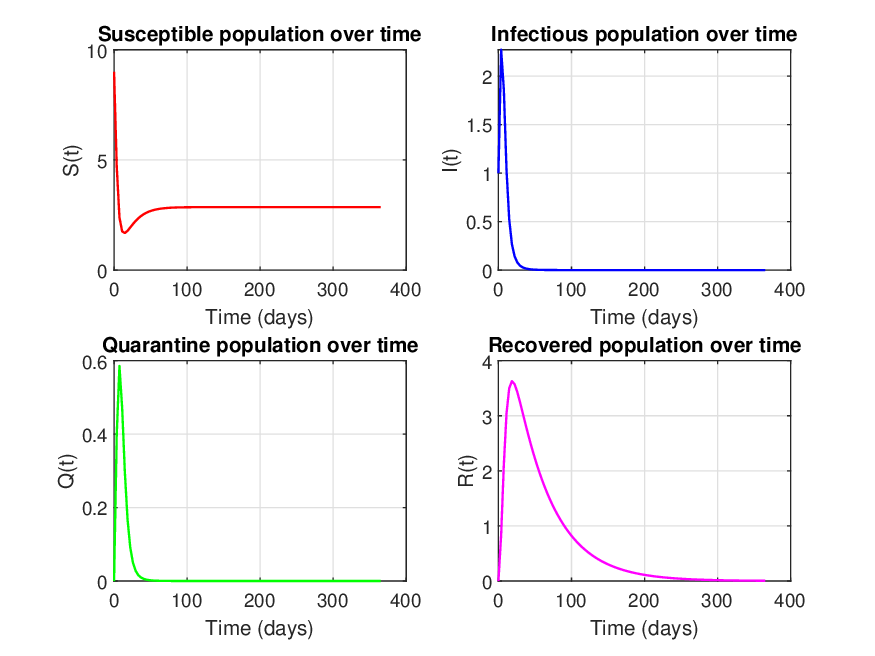}
	\caption{\label{fig:4}Variational curves of $S$, $I$, $Q$, and $R$}
\end{figure}

Now, in system (\ref{sist_vac}), let $\gamma = 0.1$, $\mu = 0.02$, $\rho = 0.3$, $\epsilon = 0.1$, $\eta = 0.2$, $\Delta = 0.2$, and $v = 0.05$. When $\alpha=0.2 $ we have $R_{0} = 1.7857 > 1$, and with the initial condition $(9,1,0,0)$ the solution converge to the endemic disease solution $(1.6,0.275,0.0859,2.6660)$.  The numerical simulation is shown in the figure $\ref{fig:5}$. From Theorem \ref{theorem 4.4}, we notice that $E^{*}$ is globally asymptotically stable.
\begin{figure}[h!]
	\centering
	\includegraphics[width=0.8\textwidth]{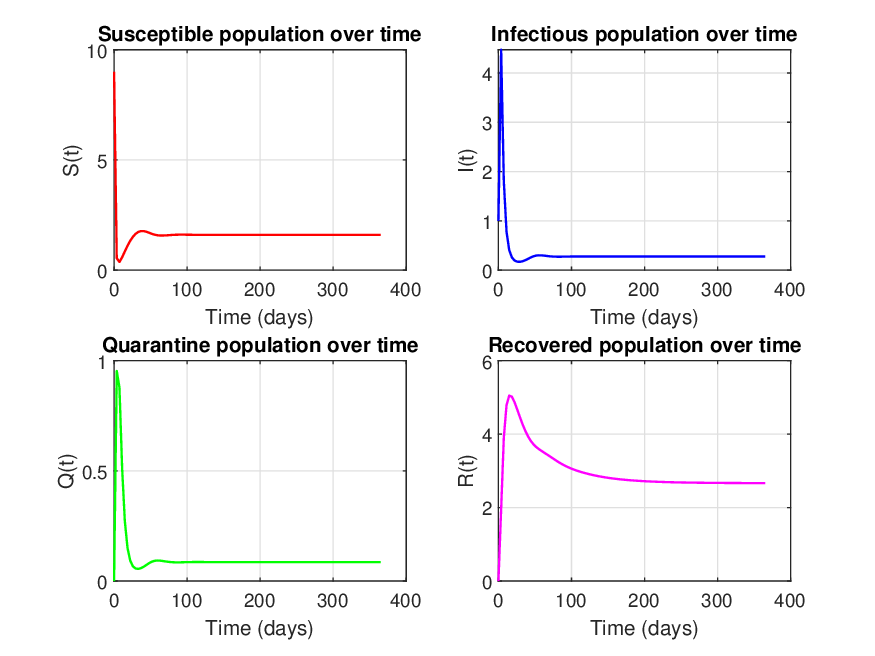}
	\caption{\label{fig:5}Variational curves of $S$, $I$, $Q$, and $R$}
\end{figure}

Additionally, we set the same initial conditions and parameters as in the previous simulation, and we get the following examples. The quarantine-free $(\eta=0)$ and vaccination-free $(v=0)$ model's reproduction number $R_{0}$ is $16.6667>1$, as shown in the numerical simulation in the figure $\ref{fig:6}$. The reproduction number $R_{0}$ vaccination-free $(v = 0)$ model is $R_{0}=6.2500>1$, the numerical simulation is shown in Figure $\ref{fig:7}$. The reproduction number $R_{0}$ quarantine-free $(\eta = 0)$ model is $R_{0}=4.7619>1$, the numerical
simulation is shown in Figure $\ref{fig:8}$.
\begin{figure}[h!]
	\centering
	\includegraphics[width=0.8\textwidth]{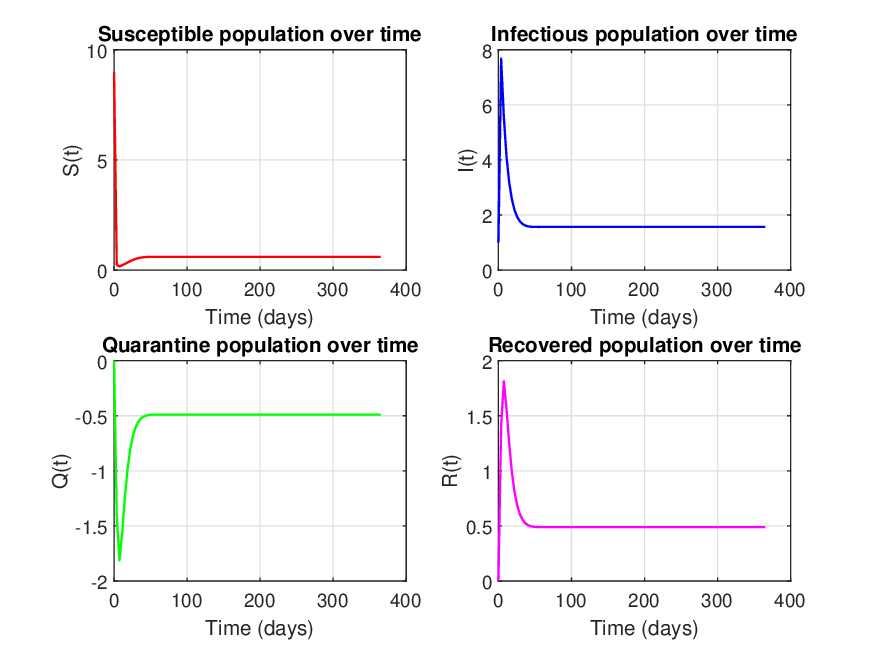}
	\caption{\label{fig:6}Variational curves of $S$, $I$, $Q$, and $R$ with $R_{0}=16.6667>1$}
\end{figure}

\begin{figure}[h!]
	\centering
	\includegraphics[width=0.8\textwidth]{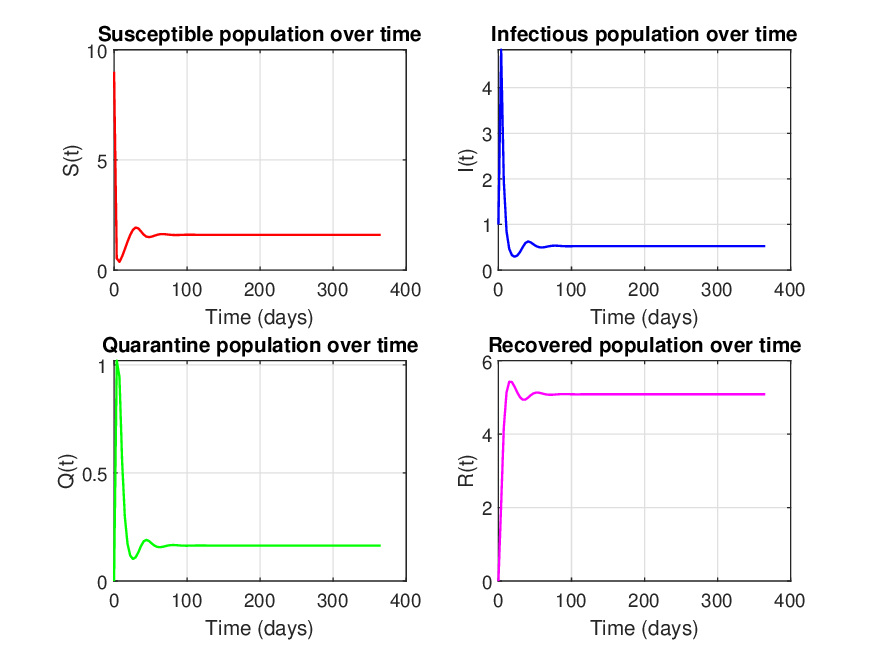}
	\caption{\label{fig:7}Variational curves of $S$, $I$, $Q$, and $R$ with $R_{0}=6.2500$ and $v=0$}
\end{figure}

\begin{figure}[h!]
	\centering
	\includegraphics[width=0.8\textwidth]{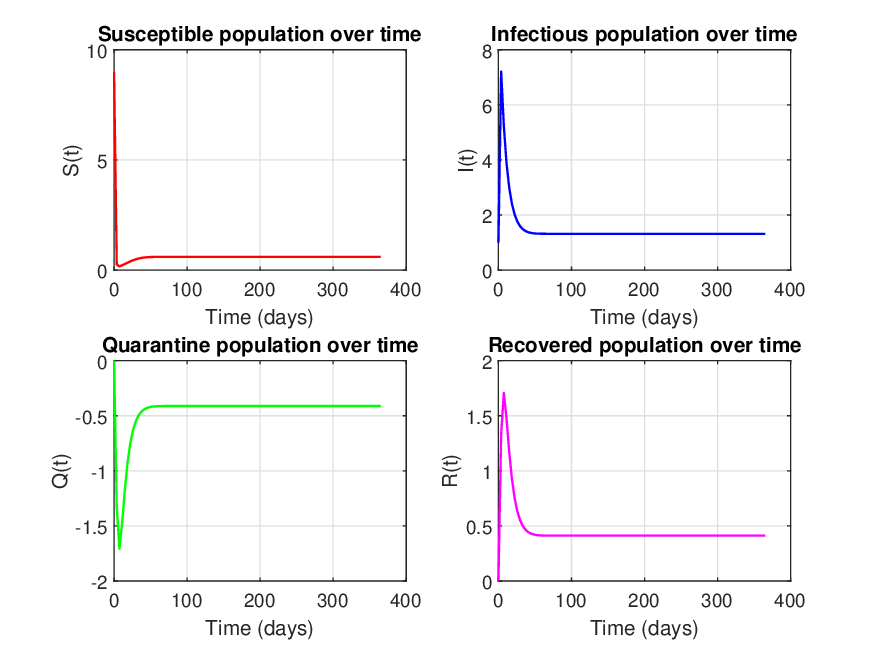}
	\caption{\label{fig:8}Variational curves of $S$, $I$, $Q$, and $R$ with $R_{0}=4.7619$ and $\eta =0$}
\end{figure}

Now, we have the evolution of the state variables without the presence of the optimal control strategy. Considering the Riccati equation $(\ref{eq.riccati})$, let $T=30$, $\gamma = 0.1$, $\mu = 0.02$, $\rho = 0.3$, $\epsilon = 0.1$, $\eta = 0.2$, $\Delta = 0.2$, and $v = 0.05$. When $\alpha=0.2$ we have $R_{0} = 1.7857 > 1$, and with the initial condition $(9,1,0,0)$,  the numerical simulation is shown in Figure $\ref{fig:9}$.

\begin{figure}[h!]
	\centering
	\includegraphics[width=0.7\textwidth]{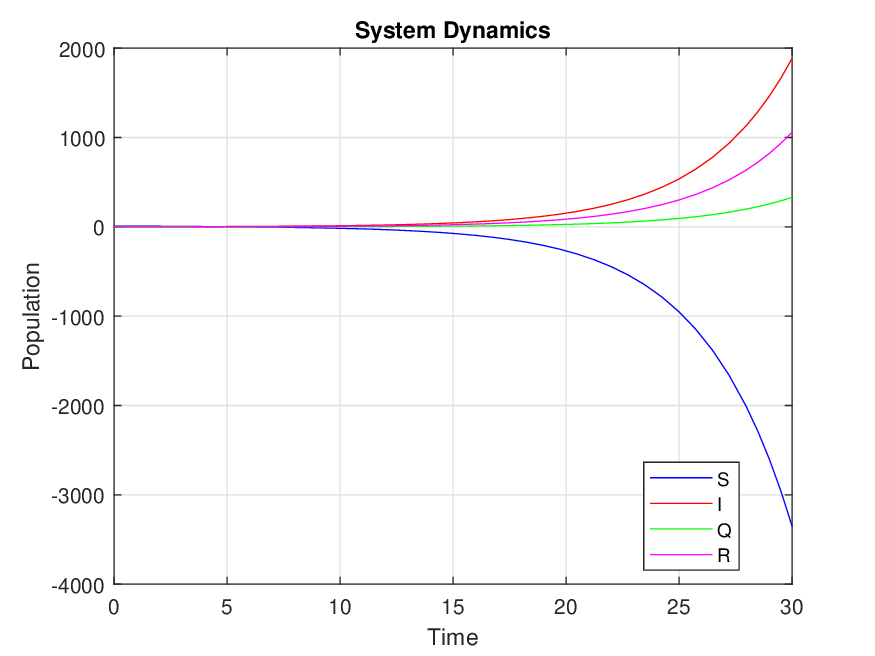}
	\caption{\label{fig:9}Variational curves of $S$, $I$, $Q$, and $R$  without optimal control }
\end{figure}

In this next simulation, we will perform numerical simulations for an optimal control strategy given by the theorem $\ref{theorem 7.1}$. Considering the Riccati equation $(\ref{eq.riccati})$, let $T=180$, $\gamma = 0.1$, $\mu = 0.02$, $\rho = 0.3$, $\epsilon = 0.1$, $\eta = 0.2$, $\Delta = 0.2$, and $v = 0.05$ and $\lambda=0.5$. When $\alpha=0.2 $ we have $R_{0} = 1.7857 > 1$, and with the initial condition $(9,1,0,0)$,  the numerical simulation is shown in Figure $\ref{fig:10}$.

\begin{figure}[h!]
	\centering
	\includegraphics[width=0.7\textwidth]{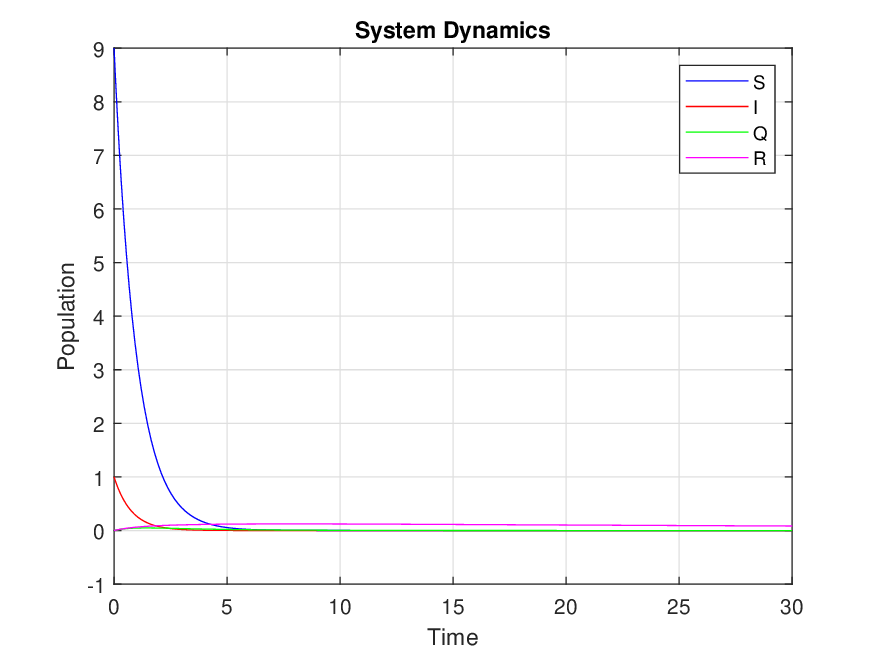}
	\caption{\label{fig:10}Variational curves of $S$, $I$, $Q$, and $R$ with whith Optimal Control}
\end{figure}

\Section{Results}
In this section we will discuss and analyze the patterns of the spread waves of Covid-19 presented in the numerical simulations previously. To do this analysis we will take as a parameter the basic reproduction number, $R_{0}$. 
We have from  the Theorem $\ref{theorem 4.1}$ that if $R_{0}<1$ the solutions converge to the disease-free equilibrium of the system $(\ref {sist_vac})$, we can see in figure $\ref{fig:4}$. This means that when the number of newly infected $I(t)$ is zero in the system, the disease-free equilibrium point occurs. This can occur when the number of individuals susceptible to $S(t)$ is small as a result of vaccination or when the amount of infected $I$ recovered is small enough to prevent the disease from continuously spreading in the population. This confirms the fact that the system is also controllable, see Theorem $\ref{theorem 5.2}$, besides, a system is controllable if, and only if it is stabilized, see \cite{micu}.
On the other hand, when $R_{0}>1$ we have the system converge to endemic equilibrium solution of the system $ (1.6,0.275,0.0859,2.6660)$, see figure $\ref{fig:5}$. In the system $(\ref{sist_vac})$, the endemic stabilization point occurs when the number of people entering the $I(t)$ compartment is equal to the numbers of people leaving the compartment for reasons of $Q(t)$ quarantine, $R(t)$ recovery, or death. 
In other words, at the endemic equilibrium point, the number of new cases is equal to the numbers of recovered or quarantined cases, thusining a balance in the amount of infected cases in the population over time. Actually, 
$$\alpha S I=\mu I+ \gamma I + \eta I=0.088.$$ 

Note that $R_{0}$ depends as much on $v$ as on $\eta$, note that $\frac{\partial R_ {0}}{\partial v}> \frac {\partial R_{ 0}}{ \partial \eta}$, in fact: 
\begin{equation*}
	\frac{\partial R_{0}}{\partial v}=\frac {\Delta \alpha}{(\mu+v)^{2}(\gamma +\mu + \eta)}
\end{equation*}
and
\begin{equation}
	\frac{\partial R_{0}}{\partial \eta}=\frac {\Delta \alpha}{(\mu+v)(\gamma +\mu + \eta)^{2}}
\end{equation}
where $\frac{\partial R_{0}}{\partial v}=1.0449
$ and $\frac{\partial R_{0}}{\partial \eta}=0.0234$.
Continuing our analysis, we can see in the $\ref{fig:7}$ simulation, that when the system presents $v=0$ and $\eta\neq 0$ the number of reproductions increases approximately $1.3$ more than when considered $ \eta=0 $ and $v$ with a 5 percent vaccination rate, see the simulation $\ref{fig:8}$. We notice that the number of basic reproduction increases very quickly when we consider $v=0$ and $\eta =0$, see $\ref{fig:6}$, where we do not consider any kind of control over the spread of the disease. From this we conclude that the presence of the vaccine is more effective as a control strategy than just control of the population of infectious individuals.

From the analyses we note that isolating infected individuals can be considered as a disease spread control strategy, but, as our main objective is to verify the effectiveness of the vaccine use strategy, we will take the isolation rate of infected persons constantly equal to $\eta$ and we will vary the population's vaccination rate for optimal control simulations.
Finally we show the optimal control, I take into account the $\ref{theorem 7.1}$ theorem that provides us with an optimal control strategy $u^{\ast}=-R^{-1}B^{ \mathbf{T}}\lambda$ of which minimizes the functional cost $(\ref {eq9})$ that gives us the performance of the system over time. Considering $R_{0}>1$, we have here a context of spread of the disease, so, we analyze this scenario without the presence of an optimal control strategy, see figure $\ref{fig:9}$, here we can see how the state variables are unstable, it is only possible to obtain a stability of the system with the increase of the percentage of the vaccinated population, i.e., from $35\%$ of the immunized population that we managed to keep the spread curves stable. On the other hand, when we use the optimum control strategies $u^{\ast}$ it is possible to see a better system performance and consequently a reduction in the cost, in the sense that, we can the same results with less effort in the application of control, see figure $\ref{fig:10}$. Therefore, the vaccine presents itself as a system state refueling control strategy, effective in combating the spread of Covid-19.

\Section{Conclusion}
In this article, we study the behavior of COVID-19 spread by a compartimental SIQR model, with vaccination as the main strategy of prevention.  We studied the balance points of the model and its stability, controlability, and optimal control for the system, finally presented some numerical simulations to corroborate the theoretical results.  From balance points and stability analysis, we learned about two constant solutions, one was the disease disappear quickly and the other continued to infect a small portion of the population, depending on the value of $R_{0}$.   Using the simulations we showed the behavior of the system (\ref{sist_vac}) and also learned that the presence of vaccination can decrease the basic reproductive number faster than the isolation of infectious individuals, eventually bringing the number of infections to the endemic balance. By applying optimal control strategies it was possible to optimize the logistical costs of the vaccine and reach the endemic equilibrium more quickly.
We conclude this work by emphasizing that, for future work, this study can be carried out for the same models combinations of varied control strategies, improving our knowledge about the behavior of this type of system and also our understanding of infectious diseases.

\section*{Acknowledgement}

To the authors of the papers that were bibliographic references.

\bibliography{arqref}

\end{document}